\documentclass[preprint]{imsart} 

\usepackage{amsthm,amsmath,amsfonts}
\usepackage{mathdots}
\usepackage{graphicx}
\usepackage{bm}
\usepackage{mathtools}

\usepackage{latexsym}
\usepackage{dsfont}
\usepackage{bbm}
\usepackage{amssymb}
\usepackage{amsmath}
\usepackage{soul}

\RequirePackage[colorlinks,citecolor=blue,urlcolor=blue]{hyperref}

\newtheorem{Theorem}{Theorem}[section]
\newtheorem{Lemma}[Theorem]{Lemma}
\newtheorem{Cor}[Theorem]{Corollary}
\newtheorem{Prop}[Theorem]{Proposition}
\newtheorem{Rem}[Theorem]{Remark}

\startlocaldefs

\def\cF{\mathcal{F}}

   \def\fm{\mathfrak{m}}

\def\Erw{\mathbb{E}}

\def\G{\mathbb{G}}

\def\N{\mathbb{N}}

\def\Prob{\mathbb{P}}

\def\R{\mathbb{R}}

\def\eps{\varepsilon}

\def\1{\textbf{1}}
\def\3{{\ss}}

\def\eqdist{\stackrel{d}{=}}

\def\wh{\widehat}

\def\ovl{\overline}

\def\dual{{}^{\#}\hspace{-1.1pt}}

\newcommand{\ps}[1]{\hspace{-1.2pt}{\scalebox{.5}{ $#1$}}}
\newcommand{\pss}[1]{\hspace{-1.2pt}{\scalebox{.55}{ $#1$}}}

\def\elll{\ell^{\ps{\bm<}}}
\def\ellg{\ell^{\ps{\bm>}}}

\def\sg{\sigma^{\ps{\bm>}}} 
\def\sgn{\sigma_{n}^{>}}

\def\sl{\sigma^{\ps{\bm<}}} 
\def\sln{\sigma_{n}^{\ps{\bm<}}}  

\def\Ml{M^{\hspace{-.5pt}\pss{\bm<}}}   
\newcommand{\Mln}{M_{n}^{\hspace{-.5pt}\pss{\bm<}}}

\def\Sl{S^{\ps{\bm<}}}   

\def\Xl{X^{\hspace{-.5pt}\pss{\bm<}}}   
\newcommand{\Xln}{X_{n}^{\hspace{-.5pt}\pss{\bm<}}}     
\def\Ql{Q^{\hspace{-.5pt}\pss{\bm<}}}   
\newcommand{\Qln}{Q_{n}^{\hspace{-.5pt}\pss{\bm<}}}
\def\Psil{\Psi^{\hspace{-.5pt}\pss{\bm<}}}   
\newcommand{\Psiln}{\Psi_{n}^{\hspace{-.5pt}\pss{\bm<}}}

\allowdisplaybreaks[4] 

\endlocaldefs

\begin{document}

\begin{frontmatter}

\title{Recurrence and transience of random difference equations in the critical case}
\runtitle{Recurrence and transience of RDE's in the critical case}

\begin{aug}
\author{\fnms{Gerold}  \snm{Alsmeyer}\corref{}\thanksref{t1}\ead[label=e1]{gerolda@uni-muenster.de}}
\and
\author{\fnms{Alexander} \snm{Iksanov}\corref{}\thanksref{t2}\ead[label=e2]{iksan@univ.kiev.ua}}
    \thankstext{t1}{Partially funded by the Deutsche Forschungsgemeinschaft (DFG) under Germany's Excellence Strategy EXC 2044--390685587, Mathematics M\"unster: Dynamics--Geometry--Structure.}
   \thankstext{t2}{Supported by the National Research Foundation of Ukraine (project 2020.02/0014 ``Asymptotic regimes of perturbed random walks: on the edge of modern and classical probability'').}

\runauthor{G.~Alsmeyer and A.~Iksanov}
  \affiliation{University of M\"unster and Taras Shevchenko National University of Kyiv}

  \address{G.~Alsmeyer\\Inst.~Math.~Stochastics,\\ Department
of Mathematics\\ and Computer Science\\ University of M\"unster\\ Orl\'eans-Ring 10,\\ D-48149 M\"unster, Germany\\
          \printead{e1}\\
          }
  \address{A.~Iksanov\\ Faculty of Computer Science and Cybernetics\\ Taras Shevchenko National\\ University of Kyiv\\ 01601 Kyiv, Ukraine\\
          \printead{e2}}

         \end{aug}

\begin{abstract}
For i.i.d. 
random vectors $(M_{1},Q_{1}),(M_{2},Q_{2}),\ldots$ such that $M>0$ a.s., $Q\geq 0$ a.s.~and $\Prob(Q=0)<1$, the random difference equation $X_{n}=M_{n}X_{n-1}+Q_{n}$, $n=1,2,\ldots$, is studied in the critical case when the random walk with increments $\log M_{1},\log M_{2}$ is oscillating. We provide conditions for the null-recurrence and transience of the Markov chain $(X_{n})_{n\ge 0}$ by inter alia drawing on techniques developed in the related article \cite{AlsBurIks:17} for another case exhibiting the null-recurrence/transience dichotomy.
\end{abstract}



\end{frontmatter}

\section{Introduction}\label{sec:intro}

Let $(M_{1},Q_{1}),(M_{2},Q_{2}),\ldots$ be i.i.d.~$\R_{+}^{2}$-valued random vectors with generic copy $(M,Q)$, where $\R_+:=[0,\infty)$. The purpose of this article is to continue recent work \cite{AlsBurIks:17} on the recurrence/transience properties of the Markov chain $(X_{n})_{n\ge 0}$ which is recursively defined by the random difference equation (RDE)
\begin{equation}\label{chain}
X_{n}\ :=\ M_{n}X_{n-1}+Q_{n},\quad n\in\N
\end{equation}
and called \emph{RDE-chain with associated random vector $(M,Q)$} hereafter. If $X_{0}=x$, we also write $X_{n}^{x}$ for $X_{n}$, and it is generally understood that $X_{0}$ and the $(M_{n},Q_{n})$ are independent. Basic assumptions throughout this work are that
\begin{gather}
\Prob(M=0)\,=\,0,\quad\Prob(Q=0)\,<\,1\label{eq:basic assumption 1},
\shortintertext{and, most importantly,}
\liminf_{n\to-\infty}\Pi_{n}\,=\,0\quad\text{and}\quad\limsup_{n\to\infty}\Pi_{n}\,=\,+\infty\quad\text{a.s.}\label{eq:basic assumption 2}
\shortintertext{where}
\Pi_{0}\,:=\,0\quad\text{and}\quad \Pi_{n}\,:=\,\prod_{k=1}^{n}M_{k}\quad\text{for }n\in\N.\nonumber
\end{gather}

Condition \eqref{eq:basic assumption 2}, which particularly holds true if
\begin{equation}\label{eq:ElogM=0}
\Erw\log M\,=\,0\quad\text{and}\quad\Prob( M=1)\,<\,1,
\end{equation}
is often referred to as the \emph{critical case} because it marks the interface between two quite different situations: the \emph{contractive case} $\lim_{n\to\infty}\Pi_{n}=0$ a.s. when the RDE-chain is positive recurrent under some mild additional conditions on $(M,Q)$, see \cite[Thm.~2.1]{GolMal:00}, and the \emph{divergent case} $\lim_{n\to\infty}\Pi_{n}=+\infty$ a.s. when the chain is typically transient. The latter can be seen from representation \eqref{eq:X_n and its dual} given below.

\vspace{.1cm}
Let us also point out that, as $M,Q$ are nonnegative, \eqref{eq:basic assumption 1} and \eqref{eq:basic assumption 2} further imply the nondegeneracy condition
\begin{equation}\label{eq:basic assumption 3}
\Prob(Mc+Q=c)\,<\,1\text{ for all }c\in\R.
\end{equation}
For a proof, notice that \eqref{eq:basic assumption 2} entails $\Prob(M<1)>0$ and $\Prob(M>1)>0$. Therefore, $Mc+Q=c$ for some $c\in\R$ would lead to the impossible conclusion that either $Q=0$ a.s., which is ruled out by \eqref{eq:basic assumption 1}, or
\begin{align*}
\Prob(Q<0)\ =\ \Prob(c(1-M)<0)\ =\ \left.
\begin{cases}
\Prob(M>1),&\text{if }c>0\\
\Prob(M<1),&\text{if }c<0
\end{cases}\right\}\ >\ 0.
\end{align*}

As usual, we put $\log_{+}x:=\log(x\vee 1)$ and $\log_{-}x:=-\log(x\wedge 1)$ for $x>0$.
Assuming \eqref{eq:ElogM=0} and, furthermore,
\begin{equation}\label{eq:Lyapunov condition}
\Erw\log^{2+\epsilon}M\,<\,\infty\quad\text{and}\quad\Erw\log_{+}^{2+\epsilon}Q\,<\,\infty
\end{equation}
for some $\epsilon>0$, Babillot et al. \cite{BabBouElie:97} showed more than twenty years ago that $(X_{n})_{n\ge 0}$ is null recurrent and possesses a unique (up to scalars) stationary Radon measure. Both intuitively and from their provided proof, one can expect that Condition \eqref{eq:Lyapunov condition} is far from being necessary. In view of the large number of publications on RDE's during the last decade, see the recent monographs by Buraczewski et al. \cite{BurDamMik:16} and Iksanov \cite{Iksanov:2016} for surveys, 
it appears to be surprising that the result has apparently not been improved until today. Such improvements are now provided by Theorems \ref{thm:main 1} and \ref{thm:main 2}, which are our main results and stated below after some further notation and relevant information.

\vspace{.1cm}
Put $S_{0}:=0$ and
$$ S_{n}\ :=\ \log\Pi_{n}\ =\ \sum_{k=1}^{n}\log M_{k}\quad\text{for }n\in\N. $$
In the critical case, $(S_{n})_{n\ge 0}$ forms an ordinary \emph{oscillating} random walk, i.e.
\begin{equation*}
\liminf_{n\to-\infty}S_{n}\,=\,-\infty\quad\text{and}\quad\limsup_{n\to\infty}S_{n}\,=\,+\infty\quad\text{a.s.}
\end{equation*}
The associated strictly descending ladder epochs, defined by $\sl_{0}:=0$ and, recursively,
\begin{equation}\label{eq:def ladder epochs}
\sln\ :=\ \inf\left\{k>\sl_{n-1}:S_{k}-S_{\sl_{n-1}}<0\right\},\quad n\in\N
\end{equation}
are then a.s. finite with infinite mean, i.e.~$\Erw\sl=\infty$ for $\sl:=\sl_{1}$. Regarding the associated first ladder height $S_{\sl}$, we note that $\Erw\log_{-}^{p+1}M<\infty$ for $p>0$
ensures
\begin{equation}\label{eq:ES_sl^p finite}
\Erw |S_{\sl}|^{p}\,<\,\infty,
\end{equation}
see \cite[p.~250]{Doney:80}. In particular, $\Erw\log_{-}^{2}M<\infty$ is sufficient for
\begin{equation}\label{eq:ES_sl finite}
\kappa\ :=\ \Erw |S_{\sl}|\,<\,\infty.
\end{equation}
Recall that $(S_{n})_{n\ge 0}$ satisfies the \emph{Spitzer condition} if, for some $0\le\rho\le 1$,
\begin{equation}\label{eq:Spitzer condition}
\lim_{n\to\infty}\frac{1}{n}\sum_{k=1}^{n}\Prob(S_{k}<0)\ =\ \rho.
\end{equation}
The limit exists also when replacing $\Prob(S_{k}<0)$ with $\Prob(S_{k}\leqslant 0),\,\Prob(S_{k}>0)$, or $\Prob(S_{k}\geqslant 0)$. Moreover, as shown by Doney \cite{Doney:95} for $0<\rho<1$ and by Bertoin and Doney \cite{BertoinDoney:97} for $\rho\in\{0,1\}$, \eqref{eq:Spitzer condition} always implies the stronger convergence
$$ \lim_{n\to\infty}\Prob(S_{n}<0)\ =\ \rho. $$
For our purposes, a more important consequence of \eqref{eq:Spitzer condition} for $0<\rho<1$ is
\begin{equation}\label{eq:Elog sl finite}
\Erw\log\sl\,<\,\infty
\end{equation}
which follows directly from the stronger tail property
\begin{equation}\label{eq:tail sl}
\Prob(\sl>n)\ \sim \ \frac{\elll_{\rho}(n)}{\Gamma(1-\rho)\,n^{\rho}}\quad\text{as }n\to\infty,
\end{equation}
valid in this case, see \cite[Thm.~8.9.12]{BingGolTeug:89}. Here $\Gamma$ denotes the Eulerian Gamma function and $\elll_{\rho}$ a slowly varying function which may be chosen as
\begin{equation}\label{eq:def ell_rho}
\elll_{\rho}(s)\ =\ \exp\left(\sum_{n\ge 1}\frac{(1-s^{-1})^{n}}{n}\big(\rho-\Prob(S_{n}<0)\big)\right),\quad s\in (1,\infty),
\end{equation}
and thus in fact as a constant if
\begin{equation}\label{eq:Spitzer series}
\sum_{n\ge 1}\frac{1}{n}\big(\rho-\Prob(S_{n}<0)\big)
\end{equation}
is convergent.

\begin{Theorem}\label{thm:main 1}
Given an RDE-chain with associated random vector $(M,Q)$ in $\R_{+}^{2}$ such that  \eqref{eq:basic assumption 1}, \eqref{eq:basic assumption 2}, and \eqref{eq:Spitzer condition} for some $\rho\in (0,1)$ hold, suppose that
\begin{gather}
\Prob(\log Q>t|M)\ \le\ \ovl{F}(t)\quad\text{a.s.}\label{eq:upper bound tail logQ}
\intertext{for all sufficiently large $t$ and a survival function $\ovl{F}$ satisfying}
\lim_{t\to\infty}t\,\ovl{F}(t)^{\rho}\,\elll_{\rho}(1/\ovl{F}(t))\ =\ 0.\label{eq:limsup Fbar}
\end{gather}
Then the chain is null-recurrent and possesses an essentially unique invariant Radon measure.
\end{Theorem}

\begin{Rem}\rm
For a comparison of this result with the one by Babillot et al. mentioned earlier, we need to compare their condition \eqref{eq:Lyapunov condition}, which entails $\rho=\frac{1}{2}$ because $(S_{n})_{n\ge 0}$ satisfies the central limit theorem, with our condition \eqref{eq:limsup Fbar} for this $\rho$. The comparison becomes easier when noting that
\begin{equation}\label{eq:cond Lyapunov condition}
\Erw\big(\log_{+}^{2+\eps}Q|M\big)\ \le\ K\quad\text{a.s.}
\end{equation}
for some constants $\eps,K>0$ provides a sufficient condition for \eqref{eq:limsup Fbar}. Then one can see that, in essence, our result does not need an extra moment condition on $\log M$ while their result makes no assumption on the joint law of $(M,Q)$.
\end{Rem}

\begin{Theorem}\label{thm:main 2}
Given an RDE-chain with associated random vector $(M,Q)$ in $\R_{+}^{2}$ satisfying  \eqref{eq:basic assumption 1}, \eqref{eq:basic assumption 2}, \eqref{eq:ES_sl finite}, and \eqref{eq:Spitzer condition} for some $\rho\in (0,1)$, the following assertions hold true:
\begin{itemize}\itemsep2pt
\item[(a)] If $Q$ satisfies \eqref{eq:upper bound tail logQ}
for all sufficiently large $t$ and a survival function $\ovl{F}$ such that
\begin{equation}\label{eq2:limsup Fbar}
s^{*}(\ovl{F})\ :=\ \limsup_{t\to\infty}t\,\ovl{F}(t)^{\rho}\,\elll_{\rho}(1/\ovl{F}(t))\ <\ \kappa,
\end{equation}
then the chain is null-recurrent and possesses an essentially unique invariant Radon measure.
\item[(b)] If $Q$ satisfies
\begin{gather}
\Prob(\log Q>t|M)\ \ge\ \ovl{G}(t)\quad\text{a.s.}\label{eq:lower bound tail logQ}
\intertext{for all sufficiently large $t$ and a survival function $\ovl{G}$ such that}
s_{*}(\ovl{G})\ :=\ \liminf_{t\to\infty}t\,\ovl{G}(t)^{\rho}\,\elll_{\rho}(1/\ovl{G}(t))\ >\ \kappa,\label{eq:liminf Gbar}
\end{gather}
then the chain is transient.
\item[(c)] If $Q$ satisfies both \eqref{eq:upper bound tail logQ} and \eqref{eq:lower bound tail logQ} for all sufficiently large $t$ and survival functions $\ovl{F},\ovl{G}$ such that $0<s_{*}(\ovl{G})\le s^{*}(\ovl{F})<\infty$, then there exists a critical exponent
$$ p_{0}\ \in\ \left[\frac{\kappa}{s^{*}(\ovl{F})},\frac{\kappa}{s_{*}(\ovl{G})}\right] $$
such that an RDE-chain with associated random vector $(M,Q^{p})$ is recurrent for $0\le p<p_{0}$ and transient for $p_{0}<p<\infty$.
\end{itemize}
\end{Theorem}

For later use, we note that the survival functions $\ovl{F}$ and $\ovl{G}$ above may easily be modified in such a way that \eqref{eq:upper bound tail logQ} and \eqref{eq:lower bound tail logQ} remain valid while \eqref{eq2:limsup Fbar} and \eqref{eq:liminf Gbar} hold in the stronger form
\begin{gather*}
s(\ovl{F})\,:=\,\lim_{t\to\infty}t\,\ovl{F}(t)^{\rho}\,\elll_{\rho}(1/\ovl{F}(t))\,<\,\kappa\quad\text{and}\quad
s(\ovl{G})\,>\,\kappa,
\end{gather*}
respectively. The latter entails that $\ovl{F}(t)$ and $\ovl{G}(t)$ are in fact regularly varying at $\infty$ with index $-1/\rho\in (-\infty,-1)$ (see \cite[Prop.~1.5.15]{BingGolTeug:89}), and we may also assume that they are smooth, convex and with monotone derivatives for sufficiently large $t$. In particular, $\ovl{F}'(t)$ and $\ovl{G}'(t)$ are negative, increasing, concave and regularly varying with index $-(1+\rho)/\rho\in (-\infty,-2)$ for large $t$, see \cite[Thms.~1.8.2 and 1.6.3]{BingGolTeug:89}).

\vspace{.1cm}
It is also clear that the \emph{conditional} tail conditions \eqref{eq:upper bound tail logQ} and \eqref{eq:lower bound tail logQ} turn into ordinary unconditional ones if $M$ and $Q$ are independent. As for Theorem \ref{thm:main 2}, it is worthwile to give an explicit formulation of the result in this case including an improvement when $t\,\Prob(\log Q>t)^{\rho}\elll_{\rho}(1/\Prob(\log Q>t))$ converges as $t\to\infty$.

\begin{Cor}\label{cor:main 2}
Given the situation of Theorem \ref{thm:main 2}, suppose further that $M,Q$ are independent and that $\ovl{F}(t):=\Prob(\log Q>t)$, the survival function of $\log Q$, satisfies $s_{*}(\ovl{F})=s^{*}(\ovl{F})=:s(\ovl{F})\in [0,\infty]$. Then the critical exponent $p_{0}$ equals $\kappa/s(\ovl{F})$, in other words, an RDE-chain with associated random vector $(M,Q^{p})$ is recurrent for $0\le p<\kappa/s(\ovl{F})$ and transient for $\kappa/s(\ovl{F})<p<\infty$. In particular, an RDE-chain with associated $(M,Q)$ is recurrent if $s(\ovl{F})<\kappa$ and transient if $s(\ovl{F})>\kappa$.
\end{Cor}

Besides the basic assumption $\Pi_{n}\to 0$ a.s. (negative divergence of $(S_{n})_{n\ge 0}$), the conditions provided by Goldie and Maller \cite[(2.1) of Thm.~2.1]{GolMal:00} for the positive recurrence and in \cite[Thms.~3.1 and 3.2]{AlsBurIks:17} for the null-recurrence or transience of an RDE-chain with associated random vector $(M,Q)$ do only involve the unconditional distributions of $M$ and $Q$. It is therefore natural to ask whether substitutes of that kind for \eqref{eq:upper bound tail logQ} and \eqref{eq:lower bound tail logQ} may also be given here. Our last two theorems provide answers that are rather pointing in another direction. In essence, the first one provides null-recurrence under a strong condition on the relation between $M$ and $Q$ but no tail condition beyond, while the second result shows that both transience and null-recurrence may occur when the laws of $M$ and $Q$ are fixed (here to be equal) but the dependence between them varies.

\begin{Theorem}\label{thm:main 3}
An RDE-chain with associated random vector $(M,Q)$ in $\R_{+}^{2}$ satisfying \eqref{eq:basic assumption 1}, \eqref{eq:basic assumption 2}, \eqref{eq:Elog sl finite}, and
\begin{equation}\label{eq:Q<=aM vee b}
Q\,\le\,aM+b\quad\text{for some }a,b>0
\end{equation}
is null-recurrent and possesses an essentially unique invariant Radon measure.
\end{Theorem}

\begin{Theorem}\label{thm:main 4}
Let $(X_{n})_{n\ge 0}$ be an RDE-chain with associated random vector $(M,Q)$ in $\R_{+}^{2}$ satisfying \eqref{eq:basic assumption 1},\eqref{eq:ElogM=0}, and $Q\eqdist M$, where $\eqdist$ means equality in law. Suppose also that $\Erw\log_{-}^{2}M<\infty$ and that the function $L(t):=t^{1/\rho}\,\Prob(\log M>t)$ is slowly varying for some $\rho\in (\frac{1}{2},1)$ with
\begin{equation*}
\lim_{t\to\infty}L(t)\ =\ \infty.
\end{equation*}
Then the chain is null-recurrent if $Q=M$, but it is transient if $M$ and $Q$ are independent.
\end{Theorem}

The proofs of these results are presented in  Section \ref{sec:proofs main}. They combine techniques from \cite{AlsBurIks:17} and \cite{BabBouElie:97}, as for the latter, the most notable being the use of an embedded contractive RDE-chain obtained by observing the original one at the descending ladder epochs $\sl_{n}$ of $(S_{n})_{n\ge 0}$, see Sections \ref{sec:ladder chain}--\ref{sec:tail lemma}.

\section{Theoretical background and prerequisites}

Defining the random linear functions $\Psi_{n}(x):=Q_{n}+M_{n} x$ for $n\in\N$, the RDE-chain $(X_{n})_{n\ge 0}$ defined by \eqref{chain} may also be viewed as the \emph{forward iterated function system}
$$ X_{n}\ =\ \Psi_{n}(X_{n-1})\ =\ \Psi_{n}\circ\ldots\circ\Psi_{1}(X_{0}),\quad n\in\N, $$
where $\circ$ denotes as usual composition of maps, and opposed to its closely related counterpart
of \emph{backward iterations}
$$\wh{X}_{0}\ :=\ X_{0}\quad\text{and}\quad \wh{X}_{n}\ :=\ \Psi_{1}\circ\ldots\circ\Psi_{n}(X_{0}),\quad n\in\N. $$
The relation is established by the obvious fact that $X_{n}$ has the same law as $\wh{X}_{n}$ for each $n$, regardless of the law of $X_{0}$. Moreover, $\Psi_{1}\cdots\Psi_{n}$ is used as shorthand for $\Psi_{1}\circ\ldots\circ\Psi_{n}$ hereafter.

\vspace{.1cm}
Put $\R_{*}:=\R\backslash\{0\}$. Since the set of affine transformations $x\mapsto ax+b$, $(a,b)\in\R_{*}\times\R$, endowed with $\circ$ as composition law forms a non-Abelian group, which is in fact isomorphic to the group $(\G,\cdot)=(\R_{*}\times\R,\cdot)$ upon defining
$$ (a_{1},b_{1})\cdot (a_{2},b_{2})\ :=\ (a_{1}a_{2},a_{1}b_{2}+b_{1}) $$
for all $(a_{1},b_{1}),(a_{2},b_{2})\in\G$, we see that $(X_{n})_{n\ge 0}$ may also be interpreted as a (left) multiplicative random walk on $\G$.

\vspace{.1cm}
Yet another sequence associated with $(X_{n})_{n\ge 0}$ and called its \emph{dual} hereafter is defined by $\dual{X}_{0}:=X_{0}$ and
\begin{equation}\label{dual chain}
\dual{X}_{n}\ :=\ \frac{1}{M_{n}}\dual{X}_{n-1}\,+\,\frac{Q_{n}}{M_{n}}.
\end{equation}
for $n\in\N$. Plainly, $(\dual{X}_{n})_{n\ge 0}$ is an RDE-chain with associated $(M^{-1},M^{-1}Q)$ and properly defined on $\R$ whenever $\Prob(M=0)=0$ which is guaranteed by Condition \eqref{eq:basic assumption 1}. The associated backward iterations $\dual{\wh{X}}_{n}:=\dual{\Psi}_{1}\cdots\dual{\Psi}_{n}(X_{0})$ for $n\in\N$, where $\dual{\Psi}(x):=M^{-1}x+M^{-1}Q$, are given by
\begin{equation*}
\dual{X}_{n}\ =\ \Pi_{n}^{-1}X_{0}+\sum_{k=1}^{n}\Pi_{k}^{-1}Q_{k}\ =\ e^{-S_{n}}X_{0}+\sum_{k=1}^{n}e^{-S_{k}}Q_{k},
\end{equation*}
so that in particular $\dual{\wh{X}}_{n}^{0}=\sum_{k=1}^{n}\Pi_{k}^{-1}Q_{k}$. We then have the obvious relation
\begin{equation}\label{eq:X_n and its dual}
X_{n}\ =\ \Pi_{n}(X_{0}\,+\dual{\wh{X}}_{n}^{0})
\end{equation}
which will be used below to provide a very simple argument for local contractivity of $(X_{n})_{n\ge 0}$.

\vspace{.1cm}
Recall that $(X_{n})_{n\ge 0}$ is called \emph{locally contractive} if, for any compact set $K$ and all $x,y\in\R$,
\begin{equation}\label{eq: local contraction}
\lim_{n\to\infty}\big| X_{n}^{x}-X_{n}^{y}\big| \cdot\1_{\{X_{n}^{x}\in K \}}\ =\ 0\quad\text{a.s.}
\end{equation}
For critical RDE-chains with associated general $\R^{2}$-valued $(M,Q)$, the notion was introduced by Babillot et al. \cite[p.~479]{BabBouElie:97} and called \emph{global stability at finite distance}. Later, Benda, in his PhD thesis \cite{Benda:98b}, used it more systematically in the framework of general stochastic dynamical systems, see also the  recent article by Peign\'e and Woess \cite{PeigneWoess:11a} for further information. Regarding RDE-chains, the notion plays an important role also in \cite{Brofferio:03,BroBura:13,Buraczewski:07,AlsBurIks:17}.

The subsequent three results summarize the main properties of locally contractive Markov chains and have also been stated (and partially proved) in \cite{AlsBurIks:17}. The first one is actually quoted from \cite[Lemma 2.2]{PeigneWoess:11a} and states that a locally contractive chain is either transient or visits a large interval infinitely often (i.o.).

\begin{Lemma}\label{lem:1}
If $(X_{n})_{n\ge 0}$ is locally contractive, then the following dichotomy holds: either
\begin{align}
&\Prob\left(\lim_{n\to\infty}|X_{n}^{x}-x|=\infty\right)\ =\ 0\quad\text{for all }x\in\R\label{eq:3}
\shortintertext{or}
&\Prob\left(\lim_{n\to\infty}|X_{n}^{x}-x|=\infty\right)\ =\ 1\quad\text{for all }x\in\R.\label{eq:4}
\end{align}
\end{Lemma}

The chain is called \emph{recurrent} if there exists a nonempty closed set $L\subset\R$ such that $\Prob(X_{n}^{x}\in U\text{ i.o.})=1$ for every $x\in L$ and every open set $U$ that intersects $L$.
\vspace{.1cm}
A proof of the next lemma can be found in \cite[Thm.~5.8]{Benda:98b} and \cite[Prop.~2.7 and Thm.~2.13]{PeigneWoess:11a}, see also \cite[Thm.~3.3]{BabBouElie:97}.

\begin{Lemma}\label{lem:3}
If $(X_{n})_{n\ge 0}$ is locally contractive and recurrent, it possesses a unique (up to a multiplicative constant) invariant Radon measure $\nu$.
\end{Lemma}

In view of this result, $(X_{n})_{n\ge 0}$ is called \emph{positive recurrent} if $\nu(L)<\infty$ and \emph{null-recurrent}, otherwise. Equivalent conditions for the transience and recurrence of $(X_{n})_{n\ge 0}$ are listed in the next proposition which may easily be proved with the help of Lemma \ref{lem:1} and Lemma 2.3 in \cite{AlsBurIks:17}.

\begin{Prop}\label{recurrent-transient}
A locally contractive Markov chain $(X_{n})_{n\ge 0}$ on $\R$ is transient iff it satisfies one of the following equivalent 
assertions:
\begin{itemize}\itemsep1.5pt
\item[(a)] $\lim_{n\to\infty}|X_{n}^{x}|=\infty$ a.s. for all $x\in\R$.
\item[(b)] $\Prob(X_{n}^{x}\in U~{\rm i.o.})< 1$ for any bounded open $U\subset\R$ and some/all $x\in\R$.
\item[(c)] $\sum_{n\ge 0}\Prob(X_{n}^{x}\in K)<\infty$ for any compact $K\subset\R$ and some/all $x\in\R$.
\end{itemize}
On the other hand, each of the following is equivalent to the recurrence of the chain:
\begin{itemize}
\item[(a)] $\liminf_{n\to\infty}|X_{n}^x-x|<\infty$ a.s. for all $x\in\R$.\vspace{.1cm}
\item[(b)] $\liminf_{n\to\infty}|X_{n}|<\infty$ a.s.\vspace{.1cm}
\item[(c)] $\sum_{n\ge 0}\Prob\{X_{n}^{x}\in K\}=\infty$ for a nonempty compact set $K$ and some/all $x\in\R$.
\end{itemize}
\end{Prop}

\vspace{.1cm}
It was shown in \cite[Thm.~3.1]{BabBouElie:97} that any RDE-chain associated with an $\R^2$-valued random vector $(M,Q)$ satisfying \eqref{eq:ElogM=0} and \eqref{eq:Lyapunov condition} is locally contractive. Their proof hinges on a number of nontrivial potential-theoretic arguments, but simplifies considerably if $M$ and $Q$ are nonnegative as also mentioned by them, see \cite[Rem.~1 on p.~486]{BabBouElie:97}. In fact, under this restriction, the result is easily extended to any critical RDE-chain satisfying our basic assumptions.

\begin{Prop}
A critical RDE-chain $(X_{n})_{n\ge 0}$ with associated random vector $(M,Q)$ in $\R_{+}^{2}$ satisfying \eqref{eq:basic assumption 1} and \eqref{eq:basic assumption 2} is locally contractive.
\end{Prop}

\begin{proof}
Let $(X_{n}^{x})_{n\ge 0}$ be defined by \eqref{chain} with $X_{0}=x\ge 0$ and $K$ an arbitrary compact subset of $\R_+$.
Denote by $\tau_{n}$, $n\in\N$, the successive epochs when the chain visits $K$, with the usual convention that $\tau_{n}:=\infty$ if the number of visits is less than $n$. We must verify \eqref{eq: local contraction} for the given $K$ only on $E:=\{\tau_{n}<\infty\text{ for all }n\in\N\}$ because it trivially holds on the complement of  this event. Use \eqref{eq:X_n and its dual} and the boundedness of $K$ to infer that
\begin{equation}\label{eq:Pi_tau_n to 0}
\sup_{n\ge 1}\Pi_{\tau_{n}}(x+\dual{\wh{X}}_{\tau_{n}}^{0})\ =\ \sup_{n\ge 1}X_{\tau_{n}}^{x}\ <\ \infty\quad\text{on }E.
\end{equation}
Since $(\dual{X}_{n})_{n\ge 0}$ is also a critical RDE-chain satisfying \eqref{eq:basic assumption 1} and \eqref{eq:basic assumption 2} and hence \emph{not} positive recurrent by the Goldie-Maller theorem \cite[Thm.~2.1]{GolMal:00}, it follows that $\dual{\wh{X}}_{n}^{0}\uparrow\infty$ a.s. But in combination with \eqref{eq:Pi_tau_n to 0}, this further entails $\Pi_{\tau_{n}}\to 0$ a.s. on $E$ and thereupon
$$ X_{\tau_{n}}^{x}-X_{\tau_{n}}^{y}\ =\ \Pi_{\tau_{n}}(x-y)\ \xrightarrow{n\to\infty}\ 0\quad\text{a.s. on }E $$
for all $x,y\ge 0$ as required.
\end{proof}

\section{The embedded ladder RDE-chain}\label{sec:ladder chain}

Recalling from \eqref{eq:def ladder epochs} the definition of the ladder epochs $\sln$, put $\Xl_{0}:=X_{0}$ and
$$ \Xln\ :=\ X_{\sln}\ =\ \Psiln\cdots\Psil_{1}(X_{0}) $$
for $n\in\N$, where
\begin{gather}
\Psiln(x)\ :=\ \Psi_{\sln}\cdots\Psi_{\sl_{n-1}+1}(x)\ =\ \Mln x\,+\,\Qln,\label{eq:def Psilessn}
\shortintertext{and}
(\Mln,\Qln)\ :=\ \frac{\Pi_{\sln}}{\Pi_{\sl_{n-1}}}\cdot\left(1,\sum_{k=\sl_{n-1}+1}^{\sln}\frac{\Pi_{\sl_{n-1}}}{\Pi_{k}}Q_{k}\right)\label{eq:def Mlessn and Qlessn}
\end{gather}
The $(\Psiln,\Mln,\Qln)$ being again i.i.d., we infer that $(\Xln)_{n\ge 0}$ is again a RDE-chain, with associated nonnegative random vector $(\Ml,\Ql)=(\Ml_{1},\Ql_{1})$, i.e.
\begin{equation}\label{eq:def Mless and Qless}
(\Ml,\Ql)\ :=\ \Pi_{\sl}\cdot\left(1\,,\sum_{k=1}^{\sl}\Pi_{k}^{-1}Q_{k}\right)\ =\ e^{S_{\sl}}\left(1\,,\sum_{k=1}^{\sl}e^{-S_{k}}Q_{k}\right).
\end{equation}
It is called \emph{embedded ladder RDE-chain} hereafter. Since $\Ml<1$ by definition of $\sl$, it is trivially strongly contractive, and under Condition \eqref{eq:Lyapunov condition}, it further satisfies
\begin{equation}\label{eq:Elies moment result}
\Erw\log_{+}\Ql\,<\,\infty
\end{equation}
as was shown by Elie \cite[Lemma 5.49]{Elie:82}. This implies the positive recurrence of the chain and the existence of a unique stationary distribution, a fact that formed an essential ingredient in \cite{BabBouElie:97}. A somewhat different approach is used here, which embarks on the strong contractivity of the ladder RDE-chain, combines it with appropriate tail estimates for $\Ql$ instead of \eqref{eq:Elies moment result} and then draws on results recently obtained in \cite{AlsBurIks:17}. It is furnished by the subsequent lemma.

\begin{Lemma}\label{lem:recurrence ladder chain}
Given a critical RDE-generated Markov chain $(X_{n})_{n\ge 0}$  with associated random vector $(M,Q)$ in $\R_{+}^{2}$ satisfying \eqref{eq:basic assumption 1} and \eqref{eq:basic assumption 2} and embedded ladder RDE-chain $(\Xln)_{n\ge 0}$, the following equivalence holds true:
$$ (X_{n})_{n\ge 0}\text{ recurrent}\quad\Longleftrightarrow\quad (\Xln)_{n\ge 0}\text{ recurrent}. $$
\end{Lemma}

\begin{proof}
We must only show that the transience of $(\Xln)_{n\ge 0}$ implies the transience of $(X_{n})_{n\ge 0}$. Observing that
\begin{align*}
\Psi_{\sln+k}\cdots\Psi_{\sln+1}(x)\ \ge\ \frac{\Pi_{\sln+k}}{\Pi_{\sln}}\,x\ \ge\ x
\end{align*}
for all $x\in\R_{+}$, $n\in\N_{0}:=\N\cup\{0\}$ and $1\le k<\sl_{n+1}-\sln$, this follows from
$$ X_{\sln+k}\ =\ \Psi_{\sln+k}\cdots\Psi_{\sln+1}(\Xln)\ \ge\ \Xln $$
in combination with $\lim_{n\to\infty}\Xln=\infty$ a.s.
\end{proof}

All previous considerations including the lemma remain true when replacing the $\sln$ by the \emph{level $\log\gamma$ ladder epochs} $\sln(\gamma)$ for an arbitrary $\gamma\in (0,1)$, defined by
$\sl_{0}(\gamma):=0$ and, recursively,
$$ \sln(\gamma)\ :=\ \inf\{k>\sl_{n-1}(\gamma):S_{k}-S_{\sl_{n-1}(\gamma)}<\log\gamma\},\quad n\in\N. $$
The sequence $(\Xln(\gamma))_{n\ge 0}:=(X_{\sln(\gamma)})_{n\ge 0}$, which then replaces $(\Xln)_{n\ge 0}$, is a RDE-chain with associated random vector
\begin{equation}\label{eq:def Mless and Qless,gamma}
(\Ml(\gamma),\Ql(\gamma))\ :=\ \Pi_{\sl(\gamma)}\cdot\left(1\,,\sum_{k=1}^{\sl(\gamma)}\Pi_{k}^{-1}Q_{k}\right),
\end{equation}
where $\sl(\gamma):=\sl_{1}(\gamma)$. Naturally, the $(\Mln(\gamma),\Qln(\gamma))$ are defined accordingly, that is (compare \eqref{eq:def Mlessn and Qlessn})
\begin{equation}
(\Mln(\gamma),\Qln(\gamma))\ :=\ \frac{\Pi_{\sln(\gamma)}}{\Pi_{\sl_{n-1}(\gamma)}}\cdot\left(1,\sum_{k=\sl_{n-1}(\gamma)+1}^{\sln(\gamma)}\frac{\Pi_{\sl_{n-1}(\gamma)}}{\Pi_{k}}Q_{k}\right)\label{eq:def Mlessn and Qlessn,gamma}
\end{equation}
for all $n\in\N$. Note that $\Ml(\gamma)=e^{\Sl(\gamma)}<\gamma$.

\section{A threshold result}\label{sec:auxiliaries}

The subsequent proposition, needed particularly for the proof of Theorem \ref{thm:main 2}, shows that, as intuitively predictable, the family of RDE-chains $(X_{p,n})_{n\ge 0}$ defined below for $p\ge 0$ exhibits a phase transition from recurrence to transience at a critical value $p_{0}$ which, however, may be zero or infinite.

\begin{Prop}\label{prop:power threshold}
Given a sequence of i.i.d.~random vectors $(M_{n},Q_{n})_{n\ge 1}$ in $\R_{+}^{2}$ with generic copy $(M,Q)$ satisfying \eqref{eq:basic assumption 1}, \eqref{eq:basic assumption 2} and \eqref{eq:Elog sl finite}, let $(X_{p,n})_{n\ge 0}$ for $p>0$ denote the RDE-chain defined by $X_{p,0}:=0$ and
\begin{equation*}
X_{p,n}\ :=\ M_{n}X_{p,n-1}\,+\,Q_{n}^{p},\quad\text{for }n\in\N.
\end{equation*}
Then there exists $p_{0}\in [0,\infty]$ such that $(X_{p,n})_{n\ge 0}$ is transient for $p>p_{0}$ (thus never if $p_{0}=\infty$) and recurrent for $p<p_{0}\ ($thus never if $p_{0}=0)$.
\end{Prop}

The proof is based on the subsequent lemma.

\begin{Lemma}\label{lem:power threshold}
In the situation of Proposition \ref{prop:power threshold}, let $(\Mln(\gamma),\Qln(\gamma))$ for any fixed $\gamma\in (0,1)$ be given by \eqref{eq:def Mlessn and Qlessn,gamma}. Further define $X_{0}^{*}=Y_{p,0}:=0$ and
\begin{gather*}
X_{n}^{*}\ :=\ M_{n}X_{n-1}^{*}\,+\,1,\\
Y_{p,n}\ :=\ M_{n}Y_{p,n-1}\,+\,(Q_{n}\vee 1)^{p}
\end{gather*}
for $n\in\N$. Then
\begin{gather}
X_{n}^{*}\vee X_{p,n}\ \le\ Y_{p,n}\ \le\ X_{n}^{*}+X_{p,n},\label{eq:inequality Y_p,n}\\
0\ \le\ Y_{p,\sln(\gamma)}-X_{p,\sln(\gamma)}\ \le\ \frac{1}{1-\gamma}\label{eq:inequality Y_p,sln}
\end{gather}
for each $n\in\N_{0}$, and the recurrence of $(X_{p,n})_{n\ge 0}$ and $(Y_{p,n})_{n\ge 0}$ are equivalent.
\end{Lemma}

\begin{proof}
Since \eqref{eq:inequality Y_p,n} follows by a straightforward induction, we turn directly to \eqref{eq:inequality Y_p,sln} and prove inductively that
$$ 0\ \le\ Y_{p,\sln(\gamma)}-X_{p,\sln(\gamma)}\ \le\ \sum_{k=0}^{n-1}\gamma^{k} $$
for all $n\in\N$. For $n=1$, this follows from
$$ Y_{p,\sl(\gamma)}-X_{p,\sl(\gamma)}\ =\ (\Ql(\gamma)\vee 1)^{p}-{\Ql(\gamma)}^{p}\ =\ (1-{\Ql(\gamma)}^{p})^{+}\ \in [0,1]. $$
Assuming it be true for arbitrary $n$, we obtain
\begin{align*}
0\ &\le\ Y_{p,\sl_{n+1}(\gamma)}-X_{p,\sl_{n+1}(\gamma)}\\
&=\ \Mln(\gamma)\big(Y_{p,\sl_{n}(\gamma)}-X_{p,\sl_{n}(\gamma)}\big)\,+\,(1-{\Ql_{n+1}(\gamma)}^{p})^{+}\ \le\ \gamma\sum_{k=0}^{n-1}\gamma^{k}\,+\,1
\end{align*}
and thus the desired result.

It remains to prove the final equivalence statement. By \eqref{eq:inequality Y_p,sln}, the recurrence of the two level $\gamma$ ladder RDE-chain $(X_{p,\sln(\gamma)})_{n\ge 0}$ and $(Y_{p,\sln(\gamma)})_{n\ge 0}$ are obviously equivalent. Hence we arrive at the desired conclusion because, by Lemma \ref{lem:recurrence ladder chain}, the joint recurrence of $(X_{p,n})_{n\ge 0}$ and $(Y_{p,n})_{n\ge 0}$ is equivalent to the joint recurrence of their aforementioned respective ladder RDE-chains.
\end{proof}

\begin{proof}[Proof of Proposition \ref{prop:power threshold}]
For the $Y_{p,n}$, $(p,n)\in (0,\infty)\times\N_{0}$, considered in the previous lemma, we obviously have $Y_{p,n}\le Y_{q,n}$ whenever $p<q$. Consequently, if $(Y_{q,n})_{n\ge 0}$ is recurrent, then the same holds true for $(Y_{p,n})_{n\ge 0}$. The set
$$ \left\{p>0:(Y_{p,n})_{n\ge 0}\text{ recurrent}\right\} $$
must therefore be an interval which may be empty. But the previous lemma further ensures that this set remains the same when replacing $(Y_{p,n})_{n\ge 0}$ with $(X_{p,n})_{n\ge 0}$.
\end{proof}


As one can readily check, Proposition \ref{prop:power threshold} remains valid if the criticality condition \eqref{eq:basic assumption 2} is replaced with
$\lim_{n\to\infty}\Pi_{n}=0$ a.s.~and  \eqref{eq:basic assumption 3}. Then it covers also the positive recurrent case when
\begin{equation}\label{eq:cond pos rec}
\lim_{n\to\infty}\Pi_{n}\,=\,0\ \text{  a.s. and }\ I_{Q}\,<\,\infty
\end{equation}
hold true, see \cite[Thm.~2.1]{GolMal:00},
and the \emph{divergent contractive case}, thus called and studied in \cite{AlsBurIks:17}, when
\begin{equation}\label{eq:cond div contractive}
\lim_{n\to\infty}\Pi_{n}\,=\,0\ \text{  a.s. and }\ I_{Q}\,=\,\infty.
\end{equation}
Here $I_{Q}:=\Erw J_{-}(\log_{+}Q)$ with $J_{-}(x):=x/\Erw(x\wedge\log_{-}M)$ for $x>0$ and $J_{-}(0):=0$. Having stated this, the next two propositions are easily obtained by combining Proposition \ref{prop:power threshold} with \cite[Thm.~2.1]{GolMal:00} and the main results in \cite{AlsBurIks:17}, respectively. They should be viewed as the counterparts of Theorem \ref{thm:main 2}(c) for these cases.

\begin{Prop}
Given a sequence of i.i.d.~random vectors $(M_{n},Q_{n})_{n\ge 1}$ in $\R_{+}^{2}$ with generic copy $(M,Q)$ satisfying \eqref{eq:basic assumption 1}, \eqref{eq:basic assumption 3} and \eqref{eq:cond pos rec}, the sequence $(X_{p,n})_{n\ge 0}$ is positive recurrent for all $p\ge 0$, thus $p_{0}=\infty$.
\end{Prop}

\begin{proof}
The result is immediate by \cite[Thm.~2.1]{GolMal:00} when observing that $I_{Q}<\infty$ is  equivalent to $I_{Q^{p}}<\infty$ for all $p>0$.
\end{proof}

For the corresponding result in the divergent contractive case, we define
$$ r_{*}(\ovl{F})\ :=\ \liminf_{t\to\infty}t\,\ovl{F}(t)\quad\text{and}\quad r^{*}(\ovl{F})\ :=\ \limsup_{t\to\infty}t\,\ovl{F}(t) $$
for any survival function $\ovl{F}$.

\begin{Prop}
Given a sequence of i.i.d.~random vectors $(M_{n},Q_{n})_{n\ge 1}$ in $\R_{+}^{2}$ with generic copy $(M,Q)$ satisfying \eqref{eq:basic assumption 1}, \eqref{eq:basic assumption 3}, \eqref{eq:cond div contractive}, and
\begin{equation*}
r^{*}(\ovl{F})\,<\,\infty
\end{equation*}
for $\ovl{F}(t):=\Prob(\log Q>t)$, the following assertions hold true:
\begin{itemize}\itemsep2pt
\item[(a)] If $\fm:=\Erw\log M\in (-\infty,0)$, then there exists a critical exponent
$$ p_{0}\ \in\ \left[\frac{|\fm|}{r^{*}(\ovl{F})},\frac{|\fm|}{r_{*}(\ovl{F})}\right] $$
such that $(X_{p,n})_{n\ge 0}$ is null-recurrent for all $p<p_{0}$ and transient for $p>p_{0}$.
\item[(b)] If $\fm=-\infty$ or does not exist, then $(X_{p,n})_{n\ge 0}$ is null-recurrent for all $p\ge 0$.
\end{itemize}
\end{Prop}

\begin{proof}
Noting that $\ovl{F}_{p}(t):=\Prob(\log Q^{p}>t)=\ovl{F}(t/p)$ for all $t\in\R$, we see that $r_{*}(\ovl{F}_{p})=pr_{*}(\ovl{F})$ and $r^{*}(\ovl{F}_{p})=pr^{*}(\ovl{F})$.

\vspace{.1cm}
(a) Suppose that $\fm\in (-\infty,0)$. By \cite[Thm.~3.1]{AlsBurIks:17}, we then infer the null-recurrence of $(X_{p,n})_{n\ge 0}$ if $pr^{*}(\ovl{F})<|\fm|$, and the transience if $pr_{*}(\ovl{F})>|\fm|$. The assertion about $p_{0}$ follows.

\vspace{.1cm}
(b) If $\fm=-\infty$ or does not exist, then we obtain the null-recurrence for all $p$, in the first case by another appeal to \cite[Thm.~3.1]{AlsBurIks:17} and in the second case by \cite[Thm.~3.2]{AlsBurIks:17}.
\end{proof}

\section{A tail lemma}\label{sec:tail lemma}

In order to prove our results by a look at the embedded ladder RDE-chain, we need information on the tail behavior of
$$ \log\Ql\ =\ \log\left(\sum_{k=1}^{\sl}e^{S_{\sl}-S_{k}}Q_{k}\right) $$
which satisfies the two inequalities
\begin{gather}
\log\Ql\ \le\ \log\sl\,+\,\max_{1\le k\le \sl}\,\log Q_{k},\label{eq:Qsl lower/upper bound}
\shortintertext{and}
\log\Ql\ \ge\ \max_{1\le k\le \sl}\big((\log Q_{k})+(S_{\sl}-S_{k})\big),\label{eq:Qsl lower bound}
\end{gather}
as one can readily see.

\begin{Lemma}\label{lem:tail lemma 1}
Given an RDE-chain with associated random vector $(M,Q)$ in $\R_{+}^{2}$ satisfying  \eqref{eq:basic assumption 1}, \eqref{eq:basic assumption 2}, and \eqref{eq:Spitzer condition} for some $\rho\in (0,1)$, the following assertions hold:
\begin{itemize}\itemsep2pt
\item[(a)] Condition \eqref{eq:upper bound tail logQ} for all sufficiently large $t$ and a survival function $\ovl{F}$ entails
\begin{equation*}
\limsup_{t\to\infty}t\,\Prob\left(\Ql>t\right)\ \le\ s^{*}(\ovl{F})\ \in\ [0,\infty].
\end{equation*}
\item[(b)]
If \eqref{eq:ES_sl finite} is additionally assumed, then Condition \eqref{eq:lower bound tail logQ} for all sufficiently large $t$ and a survival function $\ovl{G}$ entails
\begin{equation*}
\limsup_{t\to\infty}t\,\Prob\left(\Ql>t\right)\ \ge\ s_{*}(\ovl{G})\ \in\ [0,\infty].
\end{equation*}
\end{itemize}
Here $s^{*}(\ovl{F})$ and $s_{*}(\ovl{G})$ are as in \eqref{eq2:limsup Fbar} and \eqref{eq:liminf Gbar}, respectively.
\end{Lemma}

\begin{proof}
(a) By \eqref{eq:Qsl lower/upper bound}, we have for any $\eps\in (0,1)$ and $t>0$,
\begin{align*}
\Prob(\log\Ql>t)\ &\le\ \Prob(\log\sl>\eps t)\,+\,\Prob\left(\max_{1\le k\le \sl}\,\log Q_{k}>(1-\eps)t\right).
\end{align*}
Since $\Erw\log\sl<\infty$ (by \eqref{eq:Elog sl finite}) entails $\Prob(\log\sl>\eps t)=o(t)$ as $t\to\infty$, it suffices to show that
\begin{equation*}
\limsup_{t\to\infty}t\,\Prob\left(\max_{1\le k\le \sl}\,\log Q_{k}>t\right)\ \le\ s^{*}(\ovl{F}),
\end{equation*}
which in turn follows from
\begin{align*}
\Prob&\left(\max_{1\le k\le \sl}\,\log Q_{k}>t\right)\\
&=\ \sum_{n\ge 1}\int_{\{\sl=n\}}\Prob\left(\max_{1\le k\le n}\,\log Q_{k}>t\bigg|M_{1},\ldots,M_{n}\right)\ d\Prob\\
&=\ \sum_{n\ge 1}\int_{\{\sl=n\}}1-\prod_{k=1}^{n}\Prob(\log Q_{k}\le t|M_{k})\ d\Prob\\
&\le\ \sum_{n\ge 1}\int_{\{\sl=n\}}1-(1-\ovl{F}(t))^{n}\ d\Prob\\
&=\ 1\,-\,\Erw(1-\ovl{F}(t))^{\sl}\\
&\sim \ \ovl{F}(t)^{\rho}\,\elll_{\rho}(1/\ovl{F}(t))\quad\text{as }t\to\infty,
\end{align*}
where we have used \eqref{eq:upper bound tail logQ} for the fourth line and \cite[Cor.~8.1.7 on p.~334]{BingGolTeug:89} for the last one.

\vspace{.1cm}
(b) Without loss of generality $\ovl{G}(t)$ can be assumed to be regularly varying of index $-1/\rho$ at infinity and, for sufficiently large $t$, also smooth and convex with negative and concave derivative $\ovl{G}'(t)$ which is regularly varying of index $-(1+\rho)/\rho$. To see this, observe that \eqref{eq:liminf Gbar} provides $\liminf_{t\to\infty}th_{\rho}(\ovl{G}(t))>\kappa'$ for some $\kappa'>\kappa$, where $h_{\rho}(t)=t^{\rho}\ell_{\rho}(1/t)$ is regularly varying of index $\rho$ and thus ultimately increasing with regularly varying inverse $h_{\rho}^{-1}$ of index $1/\rho$. Hence $\ovl{G}(t)>h_{\rho}^{-1}(\kappa'/t)$ for all sufficiently large $t$ which in turn ensures that \eqref{eq:lower bound tail logQ} and \eqref{eq:liminf Gbar} remain valid for $\ovl{G}(t):=h_{\rho}^{-1}(\kappa'/t)$. But this function has the asserted form, including the additional smoothness properties when referring to \cite[Prop.~1.8.1]{BingGolTeug:89}. With $\ovl{G}$ thus chosen, we infer
\begin{align*}
\Prob&\left(\max_{1\le k\le \sl}\big((\log Q_{k})+(S_{\sl}-S_{k})\big)>t\right)\\
&=\ \sum_{n\ge 1}\int_{\{\sl=n\}}\Prob\left(\max_{1\le k\le n}\big((\log Q_{k})+(S_{n}-S_{k})\big)>t\bigg|M_{1},\ldots,M_{n}\right)\ d\Prob\\
&\ge\ \sum_{n\ge 1}\int_{\{\sl=n\}}1-\prod_{k=1}^{n}(1-\ovl{G}(t-S_{n}+S_{k}))\ d\Prob\\
&=\ 1\,-\,\Erw\exp\left(\sum_{k=1}^{\sl}\log(1-\ovl{G}(\zeta_{k}(t))\right),\quad\zeta_{k}(t)\,:=\,t-S_{\sl}+S_{k}\\
&\ge\ 1\,-\,\Erw\exp\left(-\sum_{k=1}^{\sl}\ovl{G}(\zeta_{k}(t))\right)\\
&=\ 1\,-\,\Erw\exp\left(-\ovl{G}(t)\,\sl\left[1-\frac{1}{\sl}\sum_{k=1}^{\sl}\left(1-\frac{\ovl{G}(\zeta_{k}(t))}{\ovl{G}(t)}\right)\right]\right)
\end{align*}
With the additional properties of $\ovl{G}$ and $\ovl{G}'(t)$, we further obtain for all sufficiently large $t$ that
\begin{align*}
&\frac{1}{\sl}\sum_{k=1}^{\sl}\left(1-\frac{\ovl{G}(\zeta_{k}(t))}{\ovl{G}(t)}\right)\ \le\ \frac{2}{\sl}\sum_{k=1}^{\sl}\left(1-\left(\frac{t}{\zeta_{k}(t)}\right)^{1/\rho}\right)\\
&\hspace{1.5cm}=\ \frac{2}{\sl}\sum_{k=1}^{\sl}\left(\frac{\zeta_{k}(t)^{1/\rho}-t^{1/\rho}}{t^{1/\rho}\,\zeta_{k}(t)^{1/\rho}}\right)\ \le\ \frac{2}{\sl}\sum_{k=1}^{\sl}\left(\frac{\zeta_{k}(t)-t}{\rho\,t^{1/\rho}\,\zeta_{k}(t)}\right)\\
&\hspace{1.5cm}\le\ \frac{2}{\rho\,t^{1/\rho}}\ =:\ \eps(t)
\end{align*}
and then, by another appeal to  \cite[Cor.~8.1.7 on p.~334]{BingGolTeug:89},
\begin{align*}
&\Prob\left(\max_{1\le k\le \sl}\big((\log Q_{k})+(S_{\sl}-S_{k})\big)>t\right)\\
&\hspace{1.5cm}\ge\ 1\,-\,\Erw\exp\left(-\ovl{G}(t)\,\big(1-\eps(t)\big)\sl\right)\\
&\hspace{1.5cm}\sim\ \ovl{G}(t)^{\rho}(1-\eps(t))^{\rho}\,\ell\left(\frac{1}{\ovl{G}(t)\,(1-\eps(t))}\right)\\
&\hspace{1.5cm}\sim\ \ovl{G}(t)^{\rho}\,\ell(1/\ovl{G}(t))
\end{align*}
as $t\to\infty$. This completes the proof of the lemma.
\end{proof}

\section{Proofs of main results}\label{sec:proofs main}

\subsection{Proof of Theorem \ref{thm:main 1}}
By Lemma \ref{lem:recurrence ladder chain}, it is enough to show recurrence of the ladder RDE-chain $(\Xln)_{n\ge 0}$ with associated random vector $(\Ml,\Ql)$. But since the latter chain is contractive $(\Pi_{\sln}\to 0\text{ a.s.})$, this follows directly from either \cite[Thm.~2.1]{GolMal:00} (positive recurrence) or \cite[Theorem 3.1(i)]{AlsBurIks:17} in combination with Lemma \ref{lem:tail lemma 1} (null-recurrence).\qed

\subsection{Proof of Theorem \ref{thm:main 2}} (a) Again, the result follows from Theorem 3.1(i) in \cite{AlsBurIks:17} applied to the ladder RDE-chain $(\Xln)_{n\ge 0}$ after noting that the latter is mean contractive, i.e.~$\Erw\log\Ml=\Erw S_{\sl}\in (-\infty,0)$, and $\Ql$ satisfies
$$ \limsup_{t\to\infty}t\,\Prob(\log\Ql>t)\ \le\ s^{*}(\ovl{F})\ <\ \Erw S_{\sl} $$
by Lemma \ref{lem:tail lemma 1}(a) and \eqref{eq2:limsup Fbar}.

\vspace{.1cm}
(b) Here the mean contractivity combines with
$$ \liminf_{t\to\infty}t\,\Prob(\log\Ql>t)\ \ge\ s_{*}(\ovl{G})\ >\ \Erw S_{\sl} $$
by Lemma \ref{lem:tail lemma 1}(b) and \eqref{eq:liminf Gbar}. Hence, the ladder RDE-chain is transient by Theorem 3.1(ii) in \cite{AlsBurIks:17}.

\vspace{.1cm}
(c) With $(M_{1},Q_{1}),(M_{2},Q_{2}),\ldots$ denoting i.i.d. copies of $(M,Q)$, let the RDE-chain $(X_{p,n})_{n\ge 0}$ be as defined in Proposition \ref{prop:power threshold} for $p>0$. By another use of Lemma \ref{lem:tail lemma 1}, here applied to $(X_{p,n})_{n\ge 0}$ with associated random vector $(M,Q^{p})$, we obtain
\begin{equation*}
ps_{*}(\ovl{G})\,\le\,\liminf_{t\to\infty}t\,\Prob(\log Q_{p,\sl}>t)\,\le\,\limsup_{t\to\infty}t\,\Prob(\log Q_{p,\sl}>t)\,\le\,ps^{*}(\ovl{F})
\end{equation*}
where $Q_{p,\sl}$ takes the role of $\Ql$ for the ladder RDE-chain $(X_{p,\sln})_{n\ge 0}$. To see this, note that $Q^{p}$ for any $p$ still satisfies \eqref{eq:upper bound tail logQ} and \eqref{eq:lower bound tail logQ}, but with $\ovl{F}(\cdot/p)$ and $\ovl{G}(\cdot/p)$ instead of $\ovl{F}$ and $\ovl{G}$, respectively. From the already shown parts (a) and (b), we finally infer the recurrence of $(X_{p,\sln})_{n\ge 0}$ and thus $(X_{p,n})_{n\ge 0}$ whenever $ps^{*}(\ovl{F})<\kappa$, and the transience whenever $ps_{*}(\ovl{G})>\kappa$. And so the critical exponent $p_{0}$ must lie between the asserted bounds $\kappa/s^{*}(\ovl{F})$ and $\kappa/s_{*}(\ovl{G})$.\qed

\subsection{Proof of Theorem \ref{thm:main 3}}

In view of Lemma \ref{lem:recurrence ladder chain}, it suffices to argue that the embedded ladder RDE-chain $(\Xl_{n})_{n\ge 0}$ is positive recurrent. By \eqref{eq:def Mless and Qless}, its associated random vector has here the form
\begin{equation*}
(\Ml,\Ql)\ :=\ \Pi_{\sl}\cdot\left(1,\sum_{k=1}^{\sl}\Pi_{k}^{-1}(aM_{k}+b)\right),
\end{equation*}
and since
\begin{align*}
\Ql\ =\ \Pi_{\sl}\left(a\sum_{k=0}^{\sl-1}\Pi_{k}^{-1}+b\sum_{k=1}^{\sl}\Pi_{k}^{-1}\right)\ \le\ (a+b)\sl
\end{align*}
we infer $\Erw\log_{+}\Ql<\infty$ with the help of \eqref{eq:Elog sl finite}. Since $(\Xln)_{n\ge 0}$ is clearly contractive, positive recurrence follows from \cite[Thm.~2.1]{GolMal:00} (or more general results like  \cite[Thm.~3]{Elton:90} or \cite[Thm.~1.1]{DiaconisFr:99}).\qed

\subsection{Proof of Theorem \ref{thm:main 4}}
In view of Theorem \ref{thm:main 3}, we must only prove the transience of the chain when $M$ and $Q$ are independent. As mentioned earlier, $\Erw\log_{-}^{2} M<\infty$ in combination with \eqref{eq:ElogM=0} ensures $\kappa=\Erw|S_{\sl}|<\infty$. Furthermore, the slow variation of $L(t)=t^{1/\rho}\,\ovl{F}(t)$ for $\ovl{F}(t):=\Prob(\log M>t)=\Prob(\log Q>t)$ and $\rho\in (\frac{1}{2},1)$ is then equivalent to the validity of \eqref{eq:Spitzer condition}, by \cite[Thm.~1]{Doney:77}, which in turn entails \eqref{eq:Elog sl finite}.
Finally, we arrive at the desired conclusion by invoking Theorem \ref{thm:main 2}(b) if we still show that $L(t)\to\infty$ implies
$$ s_{*}(\ovl{F})\ =\ \liminf_{t\to\infty}t\ovl{F}(t)^{\rho}\elll_{\rho}(1/\ovl{F}(t))\ =\ \infty. $$
To this end, we will actually prove that $\elll_{\rho}(s)\to\infty$ as $s\to\infty$ and point out first that, similar to \eqref{eq:tail sl}, we have
\begin{equation}\label{eq:tail sg}
\Prob(\sg>n)\ \sim\ \frac{\ellg_{1-\rho}(n)}{\Gamma(\rho)\,n^{1-\rho}}\quad\text{as }n\to\infty
\end{equation}
for the first strictly ascending ladder epoch $\sg:=\inf\{n\ge 1:S_{n}>0\}$, where
\begin{align*}
\ellg_{1-\rho}(s)\ :=\ \exp\left(\sum_{n\ge 1}\frac{(1-s^{-1})^{n}}{n}\big(1-\rho-\Prob(S_{n}>0)\big)\right),\quad s\in (1,\infty),
\end{align*}
is slowly varying and obviously related to $\elll_{\rho}$ by the identity
\begin{gather*}
\ellg_{1-\rho}(s)\ =\ \exp\left(\sum_{n\ge 1}\frac{(1-s^{-1})^{n}}{n}\,\Prob(S_{n}=0)\right)\frac{1}{\elll_{\rho}(s)},
\shortintertext{and thus}
\ellg_{1-\rho}(s)\ \sim\ \frac{\theta}{\elll_{\rho}(s)}\quad\text{as }s\to\infty.
\shortintertext{Here}
\theta\ :=\ \exp\left(\sum_{n\ge 1}\frac{1}{n}\,\Prob(S_{n}=0)\right)
\end{gather*}
is well-known to be always finite, see e.g. \cite[Cor.\ 3.3]{Spitzer:60b}. So it remains to verify that $\ellg_{1-\rho}(s)\to 0$ as $s\to\infty$. Now use another result by Doney \cite[Thm.~2]{Doney:82b} to infer that, under the assumptions of the theorem, the relation $\ovl{F}(t)\sim
t^{-1/\rho}L(t)$ is actually equivalent to the relation
\begin{equation}\label{eq2:tail sg}
\Prob(\sg>n)\ \sim\ \frac{c}{L_{1/\rho}^{*}(n)\,n^{1-\rho}}\quad\text{as }n\to\infty,
\end{equation}
for some $c>0$ and a slowly varying function $L_{1/\rho}^{*}$ which is related to $L$ by
$$L(s)^{-\rho}L_{1/\rho}^{*}(s^{1/\rho}/L(s))\ \to\ 1\quad\text{as }s\to\infty $$
and unique up to asymptotic equivalence. Since $L(s)\to\infty$, also $L_{1/\rho}^{*}(s)\to\infty$ holds, and we finally infer $\ellg_{1-\rho}(s)\to 0$ as $s\to\infty$ when combining \eqref{eq:tail sg} with \eqref{eq2:tail sg}.\qed

\bibliographystyle{amsplain}
\bibliography{StoPro}

\end{document}